\documentclass{amsart}
\usepackage{graphicx}
\usepackage{amsmath}
\usepackage{amssymb}
\usepackage{amsthm}
\usepackage{esint}

\newtheorem{theorem}{Theorem}[section]
\newtheorem{lemma}[theorem]{Lemma}
\newtheorem{definition}[theorem]{Definition}

\theoremstyle{Corollary}
\newtheorem{cor}[theorem]{Corollary}
\newtheorem{prop}[theorem]{Proposition}

\numberwithin{equation}{section}

\begin{document}

\title{Deformation of the scalar curvature and the mean curvature}

\author{Pak Tung Ho}
\address{Department of Mathematics, Sogang University, Seoul, 04107, Korea}

\address{Korea Institute for Advanced Study, Hoegiro 87, Seoul 02455, Korea}

\email{paktungho@yahoo.com.hk}

\author{Yen-Chang Huang}

\address{Department of Applied Mathematics, National University of Tainan, Tainan 700, Taiwan}

\email{ychuang@mail.nutn.edu.tw}


\date{17th January, 2020.}

\begin{abstract}
On a compact manifold $M$ with boundary $\partial M$, we study the problem of  prescribing the scalar curvature in $M$ and
the mean curvature on the boundary $\partial M$ simultaneously.
To do this, we introduce the notion of singular metric, which
is inspired by the early work of Fischer-Marsden in \cite{Fischer&Marsden}
and Lin-Yuan in \cite{lin2016deformations} for closed manifold.
We show that we can prescribe
the scalar curvature and the mean curvature simultaneously
for generic scalar-flat manifolds with minimal boundary.
We also prove some rigidity results for the flat manifolds with
totally geodesic boundary.

\end{abstract}

\maketitle

\section{Introduction}

Suppose that $M$ is a compact smooth manifold with boundary $\partial M$.
There are two types of the Yamabe problem with boundary:
Given a smooth metric $g$ in $M$,
(i) find a metric conformal to $g$ such that its scalar curvature is constant in $M$ and its mean curvature is zero
on $\partial M$;
(ii) find a metric conformal to $g$ such that its scalar curvature is zero in $M$ and its mean curvature is constant
on $\partial M$.
The Yamabe problem with boundary has been studied by many authors.
See \cite{Brendle&Chen,Escobar1,Escobar2,Marques} and the references therein.

As a generalization of the Yamabe problem with boundary, one can consider the  prescribing curvature problem on manifolds with boundary:
Given a smooth metric $g$ in a compact manifold $M$ with boundary $\partial M$,
(i) find a metric conformal to $g$ such that its scalar curvature is equal to a given smooth function in $M$
and its mean curvature is zero on $\partial M$;
(ii) find a metric conformal to $g$ such that its scalar curvature is zero in $M$ and its mean curvature is
equal to a given smooth function on $\partial M$.
In particular, when the manifold is the unit ball,
it is the corresponding Nirenberg's problem for manifolds with boundary.
These have been studied extensively by many authors,
We refer the readers to \cite{Chen&Ho,Escobar3,Ho,Xu&Zhang}
and the references therein for results in this direction.

More generally, one can consider the prescribing curvature problem on manifolds with boundary
without restricting to a fixed conformal class:
(i) given a smooth function $f$ in $M$, find a metric $g$ such that
its scalar curvature is equal to $f$ and its mean curvature is zero, i.e.
$R_g=f$ in $M$ and  $H_g=0$ on $\partial M$;
(ii) given a smooth function $h$ on $\partial M$, find a metric $g$ such that
its scalar curvature is zero and its mean curvature is equal to $h$, i.e.
$R_g=0$ in $M$ and  $H_g=h$ on $\partial M$.
This was recently studied by
Cruz-Vit{\'o}rio
in \cite{cruz2019prescribing}.

In this paper, we study the problem of  prescribing the scalar curvature in $M$ and
the mean curvature on the boundary $\partial M$ \textit{simultaneously}.
More precisely, given a smooth function $f$ in $M$ and
a smooth function $h$ on $\partial M$,
we want to find a metric $g$ such that
its scalar curvature is equal to $f$ and its mean curvature is equal to $h$, i.e.
$R_g=f$ in $M$ and  $H_g=h$ on $\partial M$.
We would like to point out that there are several results in
 prescribing the scalar curvature in $M$ and
the mean curvature on the boundary $\partial M$ simultaneously \textit{in a fixed conformal class}.
See \cite{Chen&Ruan&Sun,Chen&Sun,Cruz,Escobar,Han&Li1,Han&Li2}.
The flow approach was introduced to study this problem
in \cite{Brendle,Zhang}
for dimension $2$ and in \cite{Chen&Ho&Sun} for higher dimensions.
However, without restricted to a fixed conformal class,  there are not many  results
in prescribing the scalar curvature and
the mean curvature on the boundary simultaneously.
So our paper can be viewed as the first step to understand this problem.

In order to study the problem of  prescribing the scalar curvature and
the mean curvature simultaneously, we study the linearization of
the   scalar curvature and the mean curvature.
We introduce the notion of singular space in section
\ref{section_2}.
This notion is inspired by the early work of Fischer-Marsden in \cite{Fischer&Marsden},
which studied the linearization of the scalar curvature in closed (i.e. compact without boundary) manifolds,
and the work of Lin-Yuan in \cite{lin2016deformations}
which studied the linearization of the $Q$-curvature in closed manifolds.
In section \ref{section_3}, we will show that some geometric properties of the manifold imply that it is singular (or not singular).
We then give some examples of singular space and non-singular space in section \ref{section_example}.
In
 section \ref{section_5}, we prove some
theorems related to prescribing the scalar curvature and
the mean curvature simultaneously.
Finally, in section \ref{last_section},
we prove some rigidity results for the flat manifolds with
totally geodesic boundary. See Theorem \ref{rigidity}.

\section{Characterization of singular spaces}\label{section_2}

For a compact $n$-dimensional manifold $M$ with boundary $\partial M$, let  $\mathcal{M}$ be
the moduli space of all smooth metrics defined in $M$. Denote the map
\begin{equation}\label{defn_Psi}
\begin{matrix}
\Psi: &\mathcal{M}& \longrightarrow &\mathbb{R}\times \mathbb{R}\\
      & g&  \longmapsto &(R_g,2 H_\gamma)
\end{matrix}
\end{equation}
where  $\gamma=g|_{\partial M}$ is the metric $g$ induced on $\partial M$,
$R_g$ is the scalar curvature in $M$ and $H_\gamma$ the mean curvature on $\partial M$ with respect to $g$.

Let $\mathcal{S}_g: S_2(M) \longrightarrow C^\infty(M)\times C^\infty (M)$ be the linearization of $\Psi$ at $g$,
and let $\mathcal{S}_g^*:C^\infty(M)\times C^\infty (M)\longrightarrow S_2(M)$ be the $L^2$-formal adjoint of $\mathcal{S}_g$, where $S_2(M)$
is the space of symmetric $2$-tensors on $M$.
More precisely, for any $h\in S_2(M)$, we have
\begin{align*}
\left.\frac{d}{dt}\Psi(g+th)\right|_{t=0}=\mathcal{S}_g(h)=D\Psi_g\cdot h=(\delta R_gh, 2 \delta H_\gamma h).
\end{align*}
It was computed in \cite{Araujo} and \cite{cruz2019prescribing} that
\begin{equation}\label{0.1}
\begin{split}
\delta R_g h &= −\Delta_g (tr_g h) + div_g div_g h -\langle h, Ric_g\rangle,\\
2 \delta H_\gamma h&=[d(tr_g h) - div_gh](\nu) - div_\gamma X - \langle\gamma , h\rangle_\gamma,
\end{split}
\end{equation}
where $\nu$ is the outward unit normal to $\partial M$,
$II_\gamma$ is the second fundamental form of $\partial M$, $X$ is
the vector field dual to the one-form $\omega(\cdot) = h(\cdot, \nu)$,
$tr_g h= g^{ij} h_{ij}$ is the trace of $h$ and our
convention for the Laplacian is $\Delta_g f = tr_g(\mbox{Hess}_g f)$.
Now $\mathcal{S}_g^*$,  the $L^2$-formal adjoint of $\mathcal{S}_g$, satisfies
\begin{equation*}
\begin{split}
&\left.\frac{d}{dt}\left(\int_M R_{g+th}f_1 dV_g +2\int_{\partial M} H_{\gamma+th}f_2 dA_\gamma \right)\right|_{t=0}\\
&=\langle \mathcal{S}_g(h), (f_1, f_2)\rangle  =\langle h, S^*_g(f_1, f_2)\rangle
=\int_M(\delta R_g h)f_1 dV_g +2\int_{\partial M}(\delta H_\gamma h)f_2 dA_\gamma.
\end{split}
\end{equation*}
If we define
\begin{equation}\label{1.2}
\mathcal{S}^*_g(f):=\mathcal{S}^*_g(f,f)=(A^*_g f, B_\gamma^* f)~~\mbox{ where }f\in C^\infty(M),
\end{equation}
then it follows from (2.3) in \cite{cruz2019prescribing} that
\begin{equation}\label{1.1}
\begin{split}
A^*_g f     &=-(\Delta_gf)g +\mbox{Hess}_g f-f Ric_g, \\
B_\gamma^* f&=\frac{\partial f}{\partial \nu}\gamma - f II_\gamma.
\end{split}
\end{equation}

Inspired by the notion of $Q$-singular space
defined in \cite{lin2016deformations} (see also \cite{Fischer&Marsden}),
we have the following:

\begin{definition}
The metric $g$ is called singular if $\mathcal{S}_g^*$ defined in \eqref{1.2} is not injective, namely, $\ker(\mathcal{S}^*_g)\neq \{ 0\}$. We also refer $(M,\partial M, g, f)$
as singular space, if $0\not\equiv f\in \ker(\mathcal{S}^*_g)$.
\end{definition}

It follows from (\ref{1.1}) that $f\in \ker(\mathcal{S}_g^*)$ if and only if $f$ satisfies the equations
\begin{equation}\label{system1}
\left\{
\begin{split}
-(\Delta_g f)g+\mbox{Hess}_g f-f Ric_g &=0 \text{ in } M, \\
\displaystyle\frac{\partial f}{\partial \nu}\gamma -f II_\gamma &=0 \text{ on } \partial M.
\end{split}
\right.
\end{equation}
Taking the trace of \eqref{system1} with respect to $g$, we obtain
\begin{equation}\label{system2}
\left\{
\begin{split}
\displaystyle\Delta_g f+\frac{R_g}{n-1}f &=0 \text{ in } M, \\
\displaystyle\frac{\partial f}{\partial \nu}-\frac{H_\gamma}{n-1}f &=0 \text{ on }\partial M.
\end{split}
\right.
\end{equation}
That is to say, if $f\in \ker(\mathcal{S}_g^*)$, then $f$ must satisfy (\ref{system2}).

\section{Singular and nonsingular spaces}\label{section_3}

In this section, we show that some geometric properties of $(M,\partial M,g)$
will imply that it is singular (or not singular).

\begin{prop}\label{prop1.1}
If $R_g\leq 0$ and $H_\gamma\leq 0$ such that  one of them is not identically equal to zero, then $g$ is not singular.
\end{prop}
\begin{proof}
If $f\in \ker(\mathcal{S}_g^*)$, then   \eqref{system2} holds. Multiplying $f$ to the first equation in \eqref{system2},
integrating it over $M$ and using integration by parts, we obtain
\begin{equation}\label{1.3}
\begin{split}
0&=\int_M \Big( f\Delta_g f+\frac{R_g}{n-1}f^2 \Big)dV_g \\
&=\int_M \Big( -|\nabla_g f|^2+\frac{R_g}{n-1} f^2\Big)dV_g + \int_{\partial M} f\frac{\partial f}{\partial \nu}dA_\gamma\\
&=\int_M \Big( -|\nabla_g f|^2 +\frac{R_g}{n-1} f^2\Big)dV_g+\int_{\partial M}\frac{H_\gamma}{n-1}f^2dA_\gamma
\end{split}
\end{equation}
where we have used the second equation of \eqref{system2} in the last equality.
Since $R_g\leq 0$ and $H_\gamma\leq 0$,
it follows from (\ref{1.3}) that
$|\nabla_g f|^2\equiv 0$ in $M$, which implies that  $f\equiv c$ for some constant $c$.
Hence, (\ref{system2}) reduces to
$$\frac{R_g}{n-1}c=0\mbox{ in }M~~\mbox{ and }~~-\frac{H_\gamma}{n-1}c=0\mbox{ on }\partial M.$$
Since $R_g$ or $H_\gamma$  is not identically equal to zero by assumption,
we can conclude that $c=0$, i.e.
$f\equiv 0$.
Therefore, we have shown that $\ker(\mathcal{S}_g^*)=\{0\}$, as required.
\end{proof}

\begin{prop}\label{supseteq}
If $(M,\partial M,g)$ is Ricci-flat with totally-geodesic boundary, then $g$ is singular.
\end{prop}
\begin{proof}
By assumption, we have  $Rig_g\equiv 0$ in $M$ and $II_\gamma\equiv 0$ on $\partial M$. If we take $f$ to be any nonzero constant function defined in $M$, then it satisfies \eqref{system1}. Thus, $\ker(\mathcal{S}_g^*)\neq \{ 0\}$ and the result follows.
\end{proof}

In fact, we have the following:

\begin{prop}\label{Ricci-flat}
If $(M,\partial M,g)$ is Ricci-flat with totally-geodesic boundary, then
\begin{equation}\label{condition}
\ker(\mathcal{S}_g^*)=\{\mbox{constant}\}.
\end{equation}
\end{prop}
\begin{proof}
We have already shown that $\{\mbox{constant}\}\subseteq\ker(\mathcal{S}_g^*)$
 in Proposition \ref{supseteq}. On the other hand,  if $f\in \ker(\mathcal{S}_g^*)$, then $f$ satisfies  \eqref{system2}.
 This together with the assumption $Rig_g\equiv 0$ in $M$ and $II_\gamma\equiv 0$ on $\partial M$
implies that
\begin{align}\label{Neumann1}
\left\{
\begin{array}{rl}
\Delta_g f=0 &\text{ in }M, \\
\displaystyle\frac{\partial f}{\partial \nu}=0 &\text{ on }\partial M,
\end{array}
\right.
\end{align}
which shows that $f$ is constant.
\end{proof}

The condition (\ref{condition})
gives a characterization of
the Ricci-flat manifold with totally-geodesic boundary.

\begin{prop}\label{prop1.4}
If $f$ is a nonzero constant function lies in $\ker(\mathcal{S}_g^*)$, then $(M,\partial M,g)$ is Ricci-flat with totally-geodesic boundary.
\end{prop}
\begin{proof}
By assumption, the function $f\equiv c$ satisfies \eqref{system1}. Thus, we have
\begin{align*}
\left\{
\begin{array}{rl}
c Ric_g =0 &\text{ in }M, \\
-c II_\gamma =0 &\text{ on }\partial M.
\end{array}
\right.
\end{align*}
Since $c$ is nonzero, we have $Rig_g\equiv 0$ in $M$ and $II_\gamma\equiv 0$ on $\partial M$,
as required.
\end{proof}

\begin{prop}\label{existnonsing}
Suppose $R_g\equiv 0$ in $M$ and $H_\gamma\equiv 0$ on $\partial M$. If one of the following assumptions holds:
\begin{enumerate}
\item $Ric_g\not\equiv 0$ in $M$, i.e. $g$ is not Ricci-flat;
\item $II_\gamma\not\equiv 0$ on $\partial M$, i.e. $\partial M$ is not totally-geodesic,
\end{enumerate}
then $g$ is not singular.
\end{prop}
\begin{proof}
Let $f\in \ker(\mathcal{S}_g^*)$.
Since $R_g\equiv 0$ in $M$ and $H_\gamma\equiv 0$ on $\partial M$ by assumption, again by \eqref{system2} we have \eqref{Neumann1},
and thus $f\equiv c$ for some constant $c$.
If $c$ is nonzero, i.e. $f$ is a nonzero constant function lies in $\ker(\mathcal{S}_g^*)$,
it follows from Proposition \ref{prop1.4}
that $(M,\partial M, g)$ is Ricci-flat with totally-geodesic boundary,
which contradicts to the assumption.
Therefore, we must have $c=0$, i.e. $f\equiv 0$.
\end{proof}

We remark that a result similar to Proposition \ref{existnonsing} has been obtained in \cite{cruz2019prescribing}. See Proposition 3.3 in \cite{cruz2019prescribing}.

\begin{prop}
Suppose that
\begin{equation}\label{2.6}
Ric_g=\frac{R_g}{n}g=(n-1)g\mbox{ in }M~~\mbox{ and }~~H_{\gamma}=n-1\mbox{ on }\partial M.
\end{equation}
If $g$ is singular, then
$(M,\partial M, g)$ is isometric to
either  a spherical cap or  the standard hemisphere.
\end{prop}
\begin{proof}
Since $g$ is singular by assumption, there exists $0\not\equiv f\in \ker(\mathcal{S}_g^*)$.
Note that $f$ is not a non-constant function; otherwise,
it follows from Proposition \ref{prop1.4} that
$g$ is Ricci-flat, which contradicts to the assumption that
the scalar curvature $R_g$ is nonzero.
Since \eqref{system1} and
\eqref{system2}  hold, we can substitute \eqref{system2} into \eqref{system1} and apply (\ref{2.6}) to get
\begin{align*}
\left\{
\begin{array}{rl}
\displaystyle\mbox{Hess}_g f+f g =0& \text{ in }M, \\
\displaystyle\frac{\partial f}{\partial \nu}=f &\text{ on }\partial M.
\end{array}
\right.
\end{align*}
Now the result follows immediately from Theorem 3 in \cite{chen2019obata}.
\end{proof}

\begin{prop}\label{contains1}
Let $(M, \partial M, g)$ be an $n$-dimensional Einstein manifold
with minimal boundary, where $n\geq 3$. If $g$ is singular, then $\displaystyle\frac{R_g}{n-1}$ is an eigenvalue
of the Laplacian with Neumann boundary condition.
In this case, $\ker(\mathcal{S}_g^*)$ lies in the eigenspace of $\displaystyle\frac{R_g}{n-1}$.
In particular, $\ker(\mathcal{S}_g^*)$ is finite-dimensional.
\end{prop}
\begin{proof}
Let $0\not\equiv f \in \ker(\mathcal{S}_g^*)$.
Since $(M, \partial M, g)$ is  an $n$-dimensional Einstein manifold with $n\geq 3$,
the scalar curvature $R_g$ is constant. Moreover,  $H_\gamma=0$ on $\partial M$
by assumption.  Hence, \eqref{system2} implies that $f$ satisfies
\begin{align}\label{nuemann3}
\left\{
\begin{array}{rl}
\displaystyle\Delta_g f+\frac{R_g}{n-1}f =0& \text{ in } M, \\
\displaystyle\frac{\partial f}{\partial \nu} =0& \text{ on }\partial M.
\end{array}
\right.
\end{align}
This implies that $\displaystyle\frac{R_g}{n-1}$ is the eigenvalue
of the Laplacian with Neumann boundary condition, and $f$ is the corresponding eigenfunction.
This shows that $\ker(\mathcal{S}_g^*)$ lies in the eigenspace of $\displaystyle\frac{R_g}{n-1}$, as claimed.
\end{proof}

\begin{prop}\label{Steklov}
Suppose that $(M,\partial M,g)$ is scalar-flat with umbilical boundary of constant mean curvature. If $g$ is singular, then $\displaystyle\frac{H_\gamma}{n-1}$ is an Steklov eigenvalue. In this case, $\ker(\mathcal{S}_g^*)$ lies
in the eigenspace corresponding to the Steklov eigenvalue $\displaystyle\frac{H_\gamma}{n-1}$.
\end{prop}
\begin{proof}
Let $0\not\equiv f\in\ker(\mathcal{S}_g^*)$.
By assumption, we have
\begin{equation}\label{2.3}
R_g=0\mbox{ in }M~~\mbox{ and
}~~II_\gamma=\frac{H_{\gamma}}{n-1}\gamma\mbox{ on }\partial M,
\end{equation}
where $H_\gamma$ is constant.
It follows from (\ref{2.3}) that \eqref{system2} reduces to
\begin{align*}
\left\{
\begin{array}{rl}
\Delta_g f= 0 &\text{ in } M, \\
\displaystyle\frac{\partial f}{\partial \nu}+\frac{H_\gamma}{n-1}f=0 &\text{ on }\partial M.
\end{array}
\right.
\end{align*}
Hence,
$\displaystyle\frac{H_\gamma}{n-1}$ is the Steklov
eigenvalue, and $f$ is the corresponding eigenfunction.
In particular,
this shows that $\ker(\mathcal{S}_g^*)$ lies in the eigenspace corresponding to the Steklov eigenvalue $\displaystyle\frac{H_\gamma}{n-1}$.
This proves the assertion.
\end{proof}

We remark that a result similar to Proposition \ref{Steklov} has been obtained in \cite{cruz2019prescribing}. See Proposition 3.1 in \cite{cruz2019prescribing}.

\section{Examples}\label{section_example}

In this section, we give some examples of singular and non-singular space.\\

\textbf{Manifolds with negative Yamabe constant.}
Suppose that $(M,\partial M,g)$ is an $n$-dimensional compact manifold with boundary,
where $n\geq 3$. The Yamabe constant of $(M,\partial M,g)$ is defined as (c.f. \cite{Escobar2})
\begin{equation*}
Y(M,\partial M,g)=\inf\{E_g(u)\,|\,0<u\in C^\infty(M)\},
\end{equation*}
where
$$E(u)=\frac{\int_M\Big(\frac{4(n-1)}{n-2}|\nabla_gu|^2+R_gu^2\Big)dV_g+2\int_{\partial M}H_gu^2dA_g
}{\left(\int_{M}u^{\frac{2n}{n-2}}dV_g\right)^{\frac{n-2}{n}}}.$$
If the Yamabe constant of $(M,\partial M,g)$ is negative,
then we can find a metric $\tilde{g}$ conformal to $g$ such that
$R_{\tilde{g}}<0$ in $M$ and $H_{\tilde{g}}=0$ on $\partial M$ (c.f. Lemma 1.1. in \cite{Escobar2}).
In particular, it follows from Proposition \ref{prop1.1} that $(M,\partial M,\tilde{g})$
is not singular.

Similarly, we can define (c.f. \cite{Escobar1})
\begin{equation*}
Q(M,\partial M,g)=\inf\{Q_g(u)\,|\,0<u\in C^\infty(M)\},
\end{equation*}
where
$$Q_g(u)=\frac{\int_M\Big(\frac{4(n-1)}{n-2}|\nabla_gu|^2+R_gu^2\Big)dV_g+2\int_{\partial M}H_gu^2dA_g
}{\left(\int_{\partial M}u^{\frac{2(n-1)}{n-2}}dA_g\right)^{\frac{n-2}{n-1}}}.$$
If $Q(M,\partial M,g)<0$,
then we can find a metric $\tilde{g}$ conformal to $g$ such that
$R_{\tilde{g}}=0$ in $M$ and $H_{\tilde{g}}<0$ on $\partial M$ (c.f. Proposition 1.4 in \cite{Escobar1}).
In particular, it follows from Proposition \ref{prop1.1} that $(M,\partial M,\tilde{g})$
is not singular.\\

\textbf{Ricci-flat manifolds with totally geodesic boundary.}
Suppose that $(M,g)$ is a closed (i.e. compact without boundary) manifold
which is Ricci-flat.
Consider the product manifold $\tilde{M}=[0,1]\times M$ equipped with the product metric
$\tilde{g}=dt^2+g$. Then $\tilde{g}$ is still Ricci-flat,
and its boundary $\partial\tilde{M}=(\{0\}\times M)\cup(\{1\}\times M)$
is totally geodesic.
Therefore, it follows from Proposition \ref{supseteq}
that $(\tilde{M},\partial\tilde{M},\tilde{g})$
is singular.
For example, we can take $(M,g)$ to be any compact Calabi-Yau manifold.
It is Ricci-flat. Then $\tilde{M}=[0,1]\times M$ equipped with the product metric
$\tilde{g}=dt^2+g$ is singular.

Now suppose that $(M_0,g_0)$ is a closed manifold such that $g_0$ is flat.
Then the product manifold $\tilde{M}=[0,1]\times M$ equipped with the product metric
$\tilde{g}=dt^2+g$ is still flat and has totally geodesic boundary.
Therefore, it follows from Proposition \ref{supseteq}
that $(\tilde{M},\partial\tilde{M},\tilde{g})$
is singular.
For example, if we take $(M_0,g_0)$ to be the $n$-dimensional torus $T^n$ equipped with
the flat metric $g_0$, then $[0,1]\times T^n$ equipped with the product metric
$dt^2+g_0$ is flat and has geodesic boundary, and hence is singular.\\

\textbf{Product manifolds.}
Suppose that $(M,g)$ is a closed Riemannian manifold
which is scalar-flat but not Ricci-flat. Consider the
 product manifold $\tilde{M}=[0,1]\times M$ equipped with the product metric
$\tilde{g}=dt^2+g$. Then $(\tilde{M},\partial \tilde{M}, \tilde{g})$
is still scalar-flat but not Ricci-flat. Its boundary
$\partial\tilde{M}=(\{0\}\times M)\cup(\{1\}\times M)$
is totally geodesic, and thus, its mean curvature is zero.
It follows from Proposition \ref{existnonsing}
that $(\tilde{M},\partial \tilde{M}, \tilde{g})$ is not singular.

For example, let $S^2$ be
the $2$-dimensional unit sphere equipped with the standard metric $g_1$,
and $\Sigma$ be a $2$-dimensional compact manifold
with genus at least 2 equipped with the hyperbolic metric $g_{-1}$.
Then the product manifold $M=S^2\times \Sigma$ is a closed manifold,
and the product metric $g=g_1+g_{-1}$ has zero scalar curvature
and is not Ricci-flat.
From the above discussion, we can conclude that
$\tilde{M}=[0,1]\times S^2\times \Sigma$
equipped with the metric $dt^2+g_1+g_{-1}$ is not singular.\\

\textbf{The upper hemisphere.} Let
$$\mathbb{S}_+^n=\big\{(x_1, \cdots, x_{n+1})\in \mathbb{R}^{n+1}\big| x_1^2 +\cdots x_{n+1}^2=1, x_{n+1}\geq 0\big\}$$
be the $n$-dimensional upper hemisphere. We have the following:

\begin{prop}\label{hemisphere}
Let $(\mathbb{S}_+^n, \partial \mathbb{S}_+^n)$ be the $n$-dimensional upper hemisphere equipped with the standard metric  $g_c$,
i.e. the sectional curvature of $g_c$ is $1$, where $n\geq 3$. Then $g_c$
is  singular. Moreover,
$$\ker(\mathcal{S}_{g_c}^*)=\mbox{\emph{span}}\{x_1, \cdots, x_n \},$$
where $(x_1, \cdots, x_n, x_{n+1})$ are the coordinates of $\mathbb{S}_+^n\subset \mathbb{R}^{n+1}$.
\end{prop}
\begin{proof}
Note that $g_c$ is Einstein and the boundary $\partial \mathbb{S}_+^n$ is totally-geodesic, i.e.
\begin{equation}\label{2.1}
Ric_{g_c}=\displaystyle\frac{R_{g_c}}{n}g_c\mbox{ in }S^n_+~~\mbox{ and }~~
II_{\gamma_c}=0\mbox{ on }\partial \mathbb{S}_+^n,
\end{equation}
where
$R_{g_c}\equiv n(n-1)$. Note also that the coordinate functions $x_i$, $1\leq i\leq n$, satisfy
the following Obata-type equation: (see \cite{chen2019obata} for example)
\begin{equation}\label{2.2}
\mbox{Hess}_{g_c} x_i+x_i\,g_c=0\mbox{ in }S^n_+~~\mbox{ and }~~\frac{\partial x_i}{\partial\nu}=0\mbox{ on }\partial S^n_+.
\end{equation}
Combining (\ref{2.1}) and (\ref{2.2}), we can conclude that the coordinate functions
$x_i$, $1\leq i\leq n$, satisfy \eqref{system1}.
Thus, span$\{x_1, \cdots, x_n\}$ is contained in  $\ker(\mathcal{S}_{g_c}^*)$. In particular, $g_c$ is singular.

On the other hand,
 $x_i$, $1\leq i\leq n$, is an eigenfunction corresponding to
the eigenvalue $n$
of the Laplacian with Neumann boundary condition (this follows from taking trace of (\ref{2.2})).
In fact, it is well-known that the eigenspace is spanned by $x_i$ where $1\leq i\leq n$.
Hence, it follows from Proposition \ref{contains1} that
$$
\ker(\mathcal{S}_{g_c}^*)\subseteq \text{the eigenspace of }n=\mbox{span}\{x_1, \cdots, x_n\}.
$$
This proves the assertion.
\end{proof}

\textbf{The unit ball.} Let $$
D^n=\big\{(x_1, \cdots, x_n)\in \mathbb{R}^{n}\big| x_1^2 +\cdots x_n^2\leq 1\big\}$$
be the $n$-dimensional unit ball equipped with flat metric $g_0$. We have the following:

\begin{prop}\label{prop1.5}
The $n$-dimensional unit ball $(D^n,\partial D^n, g_0)$ equipped with the flat metric is a singular space.
Moreover, we have
$$
\ker(S_{g_0}^*)=\mbox{\emph{span}}\{ x_1, \cdots, x_n\},
$$
where $(x_1,...,x_n)$ are the  coordinates of $D^n$.
\end{prop}
\begin{proof}
There holds
\begin{equation}\label{2.4}
Ric_{g_0}\equiv 0\mbox{ in }D^n~~\mbox{
 and   }II_{\gamma_0}=\frac{H_{\gamma_0}}{n-1}\gamma_0\mbox{ on }\partial D^n,
\end{equation}
where $H_{\gamma_0}=n-1$.
We can easily check that $x_i$, $1\leq i\leq n$, satisfies
\begin{equation}\label{2.5}
\mbox{Hess}_{g_0}x_i=0\mbox{ in }D^n~~\mbox{ and }~~\frac{\partial x_i}{\partial\nu}=x_i\mbox{ on }\partial D^n.
\end{equation}
Combining (\ref{2.4}) and (\ref{2.5}), we can see that $x_i$, $1\leq i\leq n$, satisfies \eqref{system1}.
Therefore, span$\{x_1, \cdots, x_n\}\subseteq \ker(\mathcal{S}_{g_0}^*)$, which implies that $g_0$ is singular.

On the other hand, the eigenspace of the Steklov eigenvalue $1$ is spanned by
$x_i$, where $1\leq i\leq n$ (see Example 1.3.2 in \cite{Girouard&Polterovich} for example).
This together with Proposition \ref{Steklov} implies that
$\ker(\mathcal{S}_{g_0}^*)\subseteq \mbox{span}\{x_1,\cdots, x_n\}$
and the proof is completed.
\end{proof}

\section{Prescribing scalar curvature and mean curvature simultaneously}\label{section_5}

Given a Riemannian manifold with boundary $(M,\partial M,\bar{g})$, we have the following theorem was proved by Cruz and  Vit{\'o}rio
in \cite{cruz2019prescribing}.

\begin{theorem}[Theorem 3.5 in \cite{cruz2019prescribing}]\label{cruz}
Let $f=(f_1, f_2)\in L^p(M)\oplus  W^{\frac{1}{2}, p}(\partial M)$ where $p>n$.
Suppose that $\mathcal{S}_{\bar{g}}^*$ is injective. Then there exists $\eta>0$ such that if
$$\|f_1-R_{\bar{g}}\|_{L^p(M)}+\|f_2-H_{\bar{\gamma}}\|_{W^{\frac{1}{2},p}(\partial M)}< \eta,$$
then there is a metric $g\in \mathcal{M}^{2, p}$ such that $\Psi(g)=f$. Moreover, $g$ is smooth in any open set whenever $f$ is smooth.
\end{theorem}

More generally, we have the following:

\begin{theorem}\label{prescribing}
Let $f=(f_1, f_2)\in L^p(M)\oplus  W^{\frac{1}{2}, p}(\partial M)$ where $p>n$.
Define
\begin{equation}\label{kernel}
\Phi:=\left\{(f_1,f_2)\,\left|\,\int_Mf_1fdV_{\bar{g}}+\int_{\partial M}f_2fdA_{\bar{\gamma}}=0\mbox{ for all }f\in\ker\mathcal{S}_{\bar{g}}^*\right.\right\}.
\end{equation}
There exists $\eta>0$ such that if $(f_1,f_2)\in \Phi$ and
$$\|f_1-R_{\bar{g}}\|_{L^p(M)}+\|f_2-H_{\bar{\gamma}}\|_{W^{\frac{1}{2},p}(\partial M)}< \eta,$$
then there is a metric $g\in \mathcal{M}^{2, p}$ such that $\Psi(g)=f$. Moreover, $g$ is smooth in any open set whenever $f$ is smooth.
\end{theorem}
\begin{proof}
It was proved in P.5 of \cite{cruz2019prescribing} that
$A^*_{\bar{g}}$ is elliptic in $M$, and properly elliptic, and $B^*_\gamma$ satisfies the Shapiro-
Lopatinskij condition at any point of the boundary.
Thus, $\mathcal{S}^*_g$ defined in (\ref{1.2}) has injective symbol.
Hence, we have the following decomposition: (see \cite{Berger&Ebin} and \cite{Fischer&Marsden}; see also Theorem 4.1 in \cite{lin2016deformations})
\begin{equation}\label{decomposition}
C^\infty(M)\times C^\infty(\partial M)=\mbox{Im }\mathcal{S}_g\oplus\ker\mathcal{S}^*_g.
\end{equation}
Combining (\ref{kernel}) and (\ref{decomposition}), we have $\mbox{Im }\mathcal{S}_g=\Phi$.
By identifying $\Phi$ with its tangent space, we can see that
the map $\Psi$ defined in (\ref{defn_Psi})
is a submersion at $g$ with respect to $\Phi$.
We can now apply the Generilzed Inverse Function Theorem (c.f. Theorem 4.3 in \cite{lin2016deformations})
and conclude the local subjectivity of $\Phi$ at $g$.
This proves the assertion.
\end{proof}

The following theorem shows that we can prescribe the scalar curvature in $M$ and
the mean curvature on the boundary $\partial M$ simultaneously for
a generic scalar-flat manifold with minimal boundary.

\begin{theorem}\label{presc1}
Suppose that $(M, \partial M, \bar{g})$ is not a singular space such that  $R_{\bar{g}}=0$ in $M$
and $H_{\bar{\gamma}}=0$ on $\partial M$.
Then, for any given functions $f_1\in C^\infty(M)$ and $f_2\in C^\infty(\partial M)$, there exists a metric $g$ such that $R_g=f_1$
in $M$ and $H_\gamma=f_2$ on $\partial M$.
\end{theorem}
\begin{proof}
Let $f_1\in C^\infty(M)$ and $f_2\in C^\infty(\partial M)$.
 Since $M$ is compact, we can choose $L>0$ large enough such that
\begin{align}\label{inequ1}
\frac{\|f_1\|_\infty}{L}+\frac{\|f_2\|_\infty}{L^{1/2}}<\eta,
\end{align}
where $\eta>0$ is given as in Theorem \ref{cruz}.
Since $R_{\bar{g}}=0$ in $M$ and $H_{\bar{\gamma}}=0$ on $\partial M$ by assumption, the inequality \eqref{inequ1} can be written as
\begin{align}\label{ineq2}
\left\| \frac{f_1}{L}-R_{\bar{g}}\right\|_\infty +\left\| \frac{f_2}{L^{1/2}}-H_{\bar{\gamma}}\right\|_\infty < \eta.
\end{align}
We can now apply Theorem \ref{cruz} to conclude that  $R_g=\displaystyle\frac{f_1}{L}$ in $M$ and
$H_\gamma=\displaystyle\frac{f_2}{L^{1/2}}$ on $\partial M$
for some smooth metric $g$. Thus the metric $L^{-1}g$ satisfies
$$R_{L^{-1}g}=LR_g=L(\frac{f_1}{L})=f_1~~\mbox{ in }M,$$
and
$$H_{L^{-1}\gamma}=L^{1/2}H_\gamma=L^{1/2}(\frac{f_2}{L^{1/2}})=f_2~~\mbox{ on }M,$$
as required.
\end{proof}

As we have seen in  section \ref{section_example},
the product manifold
$M=[0,1]\times S^2\times \Sigma$
equipped with the metric $dt^2+g_1+g_{-1}$
is scalar-flat, has totally geodesic boundary, and is not singular,
where
$S^2$ is
the $2$-dimensional unit sphere equipped with the standard metric $g_1$,
and $\Sigma$ be a $2$-dimensional compact manifold
with genus at least 2 equipped with the hyperbolic metric $g_{-1}$.
Combining this with Theorem \ref{presc1}, we have the following:

\begin{cor}
Let $M=[0,1]\times S^2\times \Sigma$.
For any $f_1\in C^\infty(M)$ and $f_2\in C^\infty(\partial M)$,
there exists a metric $g$ such that $R_g=f_1$
in $M$ and $H_\gamma=f_2$ on $\partial M$.
\end{cor}

We also have the following:

\begin{theorem}\label{thm2}
Suppose $(M,\partial M, \bar{g})$ is Ricci-flat with totally-geodesic boundary.
For any  $(f_1, f_2)\in \Phi_0$ where
\begin{equation}\label{Phi_0}
\Phi_0:=\left\{ (f_1, f_2)\in C^\infty(M)\times C^\infty(\partial M)\left|\int_M f_1 dV_{\bar{g}}=\int_{\partial M} f_2 dA_{\bar{\gamma}}=0\right.\right\},
\end{equation}
there exists a metric $g$ such that $R_g=f_1$ in $M$ and $H_\gamma=f_2$ on $\partial M$.
\end{theorem}
\begin{proof}
If $(M,\partial M,\bar{g})$ is Ricci-flat with totally-geodesic boundary, it follows from Proposition \ref{Ricci-flat}
that
\begin{equation*}
\ker(\mathcal{S}_{\bar{g}}^*)=\{\mbox{constant}\}.
\end{equation*}
Hence,  $\Phi_0$ defined in (\ref{Phi_0}) is contained in $\Phi$ defined in (\ref{kernel}).
Let $(f_1,f_2)\in \Phi_0$. We can choose $L>0$ sufficiently large  such that
\begin{equation*}
\left\| \frac{f_1}{L}-R_{\bar{g}}\right\|_\infty +\left\| \frac{f_2}{L^{1/2}}-H_{\bar{\gamma}}\right\|_\infty < \eta
\end{equation*}
where $\eta$ is the given as in Theorem \ref{prescribing}.
Since $(f_1/L,f_2/L^{1/2})\in \Phi_0\subset\Phi$, it follows from Theorem \ref{prescribing} that
$R_g=\displaystyle\frac{f_1}{L}$ in $M$ and
$H_\gamma=\displaystyle\frac{f_2}{L^{1/2}}$ on $\partial M$
for some smooth metric $g$ closed to $\bar{G}$. Thus the metric $L^{-1}g$ satisfies
$$R_{L^{-1}g}=LR_g=L(\frac{f_1}{L})=f_1~~\mbox{ in }M,$$
and
$$H_{L^{-1}\gamma}=L^{1/2}H_\gamma=L^{1/2}(\frac{f_2}{L^{1/2}})=f_2~~\mbox{ on }M,$$
as required.
\end{proof}

As we have seen in section \ref{section_example},
for any closed Ricci-flat $(M,g)$,
the product manifold
$\tilde{M}=[0,1]\times M$ equipped with the product metric
$\tilde{g}=dt^2+g$ is Ricci-flat with totally-geodesic boundary.
Therefore, from Theorem \ref{thm2}, we immediately have the following

\begin{cor}
Suppose $(M,g)$ is a closed Ricci-flat manifold.
Let
$\tilde{M}=[0,1]\times M$ be the product manifold
equipped with the product metric $\tilde{g}=dt^2+g$.
Then, for any  $(f_1, f_2)\in C^\infty(\tilde{M},\partial\tilde{M})$ such that
$$
\int_{\tilde{M}} f_1 dV_{\tilde{g}}=0~~\mbox{ and }~~\int_{\partial \tilde{M}} f_2 dA_{\tilde{\gamma}}=0,
$$
there exists a metric $g$ such that $R_g=f_1$ in $\tilde{M}$ and $H_\gamma=f_2$ on $\partial \tilde{M}$.
\end{cor}

Next
we have the following theorem of prescribing the scalar curvature
and the mean curvature simultaneously  on the upper hemisphere.

\begin{theorem}
Let $f_1\in C^\infty(\mathbb{S}_+^n)$ and $f_2\in C^\infty(\partial \mathbb{S}_+^n)$ such that
$$\int_{\mathbb{S}_+^n} x_i f_1 dV_{g_c}=\int_{\partial \mathbb{S}_+^n} x_i f_2 dA_{\gamma_c}=0~~\mbox{ for }1\leq i\leq n.$$
Then there exists a metric $g$ such that $R_g=f_1$ in $\mathbb{S}^n_+$ and $H_\gamma=f_2$ on $\partial \mathbb{S}^n_+$.
\end{theorem}
\begin{proof}
It follows from
Proposition \ref{hemisphere} that
$\ker(\mathcal{S}_{g_c}^*)=\mbox{span}\{x_1, \cdots, x_n \}$. Hence,
the space
$$\left\{(f_1,f_2)\in C^\infty(\mathbb{S}_+^n)\times C^\infty(\partial \mathbb{S}_+^n)\left|\,\int_{\mathbb{S}_+^n} x_i f_1 dV_{g_c}=\int_{\partial \mathbb{S}_+^n} x_i f_2 dA_{g_c}=0~~\mbox{ for }1\leq i\leq n\right.\right\}$$
lies in  $\Phi$ defined in (\ref{kernel}).
Using this, we can follow the same argument as in the proof of Theorem \ref{thm2}
to finish the proof.
\end{proof}

Finally we have the following theorem of prescribing  the scalar curvature
and the mean curvature simultaneously on the unit ball.

\begin{theorem}
Given any  $f_1\in C^\infty(D^n)$ and $f_2\in C^\infty(\partial D^n)$ such that
\begin{equation}\label{condition1}
\int_{D^n}f_1 x_i dV_{g_0}=\int_{\partial D^n} f_2x_i dA_{\gamma_0}=0~~\mbox{ for any }1\leq i\leq n.
\end{equation}
Then there exists a metric $g$ such that $R_g=f_1$ in $D^n$ and $H_\gamma=f_2$ on $\partial D^n$.
\end{theorem}
\begin{proof}
It follows from Proposition \ref{prop1.5} that
$\ker(\mathcal{S}_{g_0}^*)=\mbox{span}\{x_1, \cdots, x_n \}$. Hence,
the space of all $(f_1,f_2)$ satisfying (\ref{condition1})
lies in  $\Phi$ defined in (\ref{kernel}).
Hence,  we can follow the same argument as in the proof of Theorem \ref{thm2}
to finish the proof.
\end{proof}

\section{Rigidity results}\label{last_section}

Suppose that $(M,\partial M,\bar{g},f)$ is a singular space
such that
\begin{equation}\label{4.0}
R_{\bar{g}}=0\mbox{ in }M~~\mbox{ and }~~H_{\bar{\gamma}}=0\mbox{ on }\partial M.
\end{equation}
We define the following functional:
\begin{equation}\label{4.1}
\mathcal{F}(g)=\int_M R_g fdV_{g}+2\int_{\partial M}H_\gamma fdA_{g}
\end{equation}
for $g\in \mathcal{M}$.
We have the following:

\begin{lemma}\label{lem4.1}
The metric $\bar{g}$ is a critical point of $\mathcal{F}$ defined in \eqref{4.1}.
\end{lemma}
\begin{proof}
We compute
\begin{equation*}
\begin{split}
\left.\frac{d}{dt}\mathcal{F}(\bar{g}+th)\right|_{t=0}&=
\int_M(\delta R_{\bar{g}} h)f dV_{\bar{g}} +2\int_{\partial M}(\delta H_{\bar{\gamma}} h)f dA_{\bar{\gamma}}\\
&\hspace{4mm}+\int_M R_{\bar{g}} \left.\frac{\partial}{\partial t}dV_{\bar{g}+th}\right|_{t=0}+2\int_{\partial M}H_{\bar{\gamma}} \left.\frac{\partial}{\partial t}dA_{\bar{\gamma}+th}\right|_{t=0}\\
&=\langle\mathcal{S}_{\bar{g}}(h),(f,f)\rangle
=\langle h,\mathcal{S}_{\bar{g}}^*(f,f)\rangle=0,
\end{split}
\end{equation*}
where we have used
(\ref{4.0}) and the fact that $\mathcal{S}_{\bar{g}}^*(f,f)=0$.
This proves the assertion.
\end{proof}

From now on, we suppose that $(M,\partial M,\overline{g})$ is a compact $n$-dimensional manifold
which is flat (hence is Ricci-flat) and has totally geodesic boundary.
It follows from Proposition \ref{supseteq} and Proposition
\ref{Ricci-flat} that  $\overline{g}$ is singular and we can take $f\equiv 1$.
Then the functional $\mathcal{F}$ defined in (\ref{4.1}) becomes
\begin{equation}\label{4.2}
\mathcal{F}(g)=\int_M R_g dV_{g}+2\int_{\partial M}H_\gamma dA_{\gamma}.
\end{equation}

We will prove the following rigidity theorem.

\begin{theorem}\label{rigidity}
Let $(M,\partial M,\bar{g})$ be a compact $n$-dimensional manifold
which is flat and has totally geodesic boundary.
If $g$ is sufficiently closed to $\bar{g}$ such that\\
(i) $R_g\geq 0$ in $M$ and $H_\gamma\geq 0$ on $\partial M$,\\
(ii) $g$ and $\bar{g}$ induce the same metric on $\partial M$,\\
then $(M,\partial M,g)$ is also flat and has totally geodesic boundary.
\end{theorem}

To prove Theorem \ref{rigidity}, we need have the following proposition from \cite{Brendle&Marques}:

\begin{prop}[Proposition 11 in \cite{Brendle&Marques}]\label{prop11}
Let $M$ be a compact $n$-dimensional manifold with boundary $\partial M$.  Fix a real number $p>n$.
 If
$\|g-\bar{g}\|_{W^{2,p}(M,\bar{g})}$
is sufficiently small such that $g$ and $\bar{g}$ induce the same metric on $\partial M$,
then we can find a diffeomorphism
$\varphi: M\to M$ such that $\varphi|_{\partial M}=id$ and
$h=\varphi^*(g)-\bar{g}$ is divergence-free with respect to $\bar{g}$.  Moreover,
\begin{equation}\label{bound}
\|h\|_{W^{2,p}(M,\bar{g})}\leq N\|g-\bar{g}\|_{W^{2,p}(M,\bar{g})}
\end{equation}
where $N$ is a positive constant that depends only on $M$.
\end{prop}

We are now ready to prove Theorem \ref{rigidity}.

\begin{proof}[Proof of Theorem \ref{rigidity}]
Suppose that $g$ and $\bar{g}$ are given as in Theorem \ref{rigidity}.
We can apply Proposition \ref{prop11} to get a diffeomorphism
$\varphi: M\to M$ such that $\varphi|_{\partial M}=id$,
$h=\varphi^*(g)-\bar{g}$ is divergence-free with respect to $\bar{g}$
and satisfies (\ref{bound}).
Note that
\begin{equation}\label{simplify}
h=\varphi^*(g)-\bar{g}=0~~\mbox{ on }\partial M,
\end{equation}
since $g$ and $\bar{g}$ induce the same metric on $\partial M$
and $\varphi|_{\partial M}=id$.
We compute
\begin{equation}\label{4.3}
\mathcal{F}(\varphi^*g)=\mathcal{F}(\bar{g})
+D\mathcal{F}_{\bar{g}}(h)
+\frac{1}{2}D^2\mathcal{F}_{\bar{g}}(h,h)+E_3,
\end{equation}
where $E_3$ is bounded by (see (7.11) in \cite{Case})
\begin{equation}\label{4.4}
|E_3|\leq C\|h\|_{C^0(M,\bar{g})}\int_M |\nabla_{\bar{g}} h|^2dV_{\overline{g}}
\end{equation}
for some constant $C$ depending only on $(M,\partial M,\overline{g})$,
thanks to (\ref{simplify}).
It follows from the assumption and  Lemma \ref{lem4.1} that
\begin{equation}\label{4.5}
\mathcal{F}({\bar{g}})=0~~\mbox{ and }~~D\mathcal{F}_{\bar{g}}(h)=0.
\end{equation}

We are going to compute $D^2\mathcal{F}_{\bar{g}}(h,h)$.
To this end, we have the following formula:  (see the last equation in P.124
of \cite{cruz2019prescribing})
\begin{equation*}
\begin{split}
&\int_M f\delta R_{\hat{g}}dV_{\hat{g}} +2\int_{\partial M}f\delta H_{\hat{\gamma}} hdA_{\hat{\gamma}}\\
&=\int_M\Big(-\Delta_{\hat{g}}f(tr_{\hat{g}}h)+\langle\mbox{Hess}_{\hat{g}}f,h\rangle-f\langle h,Ric_{\hat{g}}\rangle\Big)dV_{\hat{g}}\\
&\hspace{4mm}
+\int_M\Big(tr_{\hat{\gamma}}h\frac{\partial f}{\partial \nu}-f\langle II_{\hat{\gamma}},h\rangle_{\hat{\gamma}}\Big)dA_{\hat{\gamma}}
\end{split}
\end{equation*}
for any metric $\hat{g}$ and any smooth function $f$. In particular, if we take $f\equiv 1$ and $\hat{g}=\bar{g}+th$,
we have
\begin{equation*}
\begin{split}
&\int_M(\delta R_{\bar{g}+th} h)dV_{\bar{g}+th} +2\int_{\partial M}(\delta H_{\bar{\gamma}+th} h)dA_{\bar{\gamma}+th}\\
&=-\int_M\langle h,Ric_{\bar{g}+th}\rangle dV_{\bar{g}+th}
-\int_M\langle II_{\bar{\gamma}+th},h\rangle_{\bar{\gamma}+th}dA_{\bar{\gamma}+th}.
\end{split}
\end{equation*}
Differentiating it with respect to $t$, evaluating it at $t=0$ and
using the fact that $\overline{g}$ is flat with totally geodesic boundary, we obtain
\begin{equation}\label{4.6}
\begin{split}
D^2\mathcal{F}_{\bar{g}}(h,h)&=\frac{d}{dt}\left.\left(\int_M(\delta R_{\bar{g}+th} h)dV_{\bar{g}+th} +2\int_{\partial M}(\delta H_{\bar{\gamma}+th} h)dA_{\bar{\gamma}+th}\right)\right|_{t=0}\\
&=-\int_M\langle h,\left.\frac{\partial}{\partial t}(Ric_{\bar{g}+th})\right|_{t=0}\rangle dV_{\bar{g}}
-\int_M\langle \left.\frac{\partial}{\partial t}(II_{\bar{\gamma}+th})\right|_{t=0},h\rangle_{\bar{\gamma}}dA_{\bar{\gamma}}\\
&=-\int_M\langle h,\left.\frac{\partial}{\partial t}(Ric_{\bar{g}+th})\right|_{t=0}\rangle dV_{\bar{g}}
\end{split}
\end{equation}
where the last equality follows from (\ref{simplify}).
There holds (see (3.2) in \cite{lin2016deformations} for example)
\begin{equation}\label{4.7}
\left.\frac{\partial}{\partial t}(Ric_{\bar{g}+th})_{jk}\right|_{t=0}
=-\frac{1}{2}\big(\Delta_L h_{jk}+\nabla_j\nabla_k(tr_{\bar{g}}h)+\nabla_j(div_{\bar{g}} h)_k+\nabla_k(div_{\bar{g}} h)_j\big).
\end{equation}
Here the Licherowicz Laplacian acting on $h$ is defined as
\begin{equation}\label{4.8}
\Delta_Lh_{jk}=\Delta h_{jk}+2(\overset{\circ}{Rm}\cdot h)_{jk}-Ric_{ji}h^i_k-Ric_{ki}h^i_j,
\end{equation}
where the geometric quantities on the right hand side is with respect to $\bar{g}$.
Since $\bar{g}$ is flat and $div_{\bar{g}}h=0$,
it follows from  (\ref{4.6})-(\ref{4.8}) that
\begin{equation}\label{4.9.5}
D^2\mathcal{F}_{\bar{g}}(h,h)=\frac{1}{2}\int_M \Big(h_{jk}\Delta_{\bar{g}}h_{jk}+h_{jk}\nabla_j\nabla_k(tr_{\bar{g}}h)\Big) dV_{\bar{g}}.
\end{equation}
By integration by parts, (\ref{simplify})
and
the fact that $h$ is divergence-free with respect to $\bar{g}$,
we can rewrite (\ref{4.9.5}) as
\begin{equation}\label{4.9}
\begin{split}
D^2\mathcal{F}_{\bar{g}}(h,h)
&=\frac{1}{2}\int_M \Big(h_{jk}\Delta_{\bar{g}}h_{jk}-\nabla_jh_{jk}\nabla_k(tr_{\bar{g}}h)\Big) dV_{\bar{g}}=-\frac{1}{2}\int_M |\nabla_{\bar{g}}h|^2dV_{\bar{g}}.
\end{split}
\end{equation}

Now, we can combine (\ref{4.3}), (\ref{4.5}) and (\ref{4.9}) to obtain
\begin{equation}\label{4.10}
\mathcal{F}(\varphi^*g)=-\frac{1}{2}\int_M |\nabla_{\bar{g}} h|^2dV_{\bar{g}}+E_3
\end{equation}
where $E_3$ satisfies (\ref{4.4}).
By assumption (i) in Theorem \ref{rigidity} and the fact that $\varphi$
is a diffeomorphism, we have
\begin{equation}\label{4.11}
\mathcal{F}(\varphi^*g)=\mathcal{F}(g)\geq 0.
\end{equation}
Combining (\ref{4.10}) and (\ref{4.11}), we get
\begin{equation}\label{4.12}
0\leq -\frac{1}{2}\int_M |\nabla_{\bar{g}} h|^2dV_{\bar{g}}+E_3.
\end{equation}
In view of (\ref{bound}), (\ref{4.4}) and (\ref{4.12}), we can conclude
that $\nabla_{\bar{g}} h=0$ when $g$ is sufficiently closed to $\bar{g}$.
In particular, $h_{jk}$ is constant for each pair of $j,k$.
Since $h=0$ on $\partial M$ by (\ref{simplify}),
we must have $h=0$ in $M$. That is to say, $\varphi^*(g)=\tilde{g}$.
Hence, $(M,\partial M,g)$ is also flat and has totally geodesic boundary.
This finishes the proof of Theorem \ref{rigidity}.
\end{proof}

We remark that the second variation of
the functional defined in (\ref{4.2})
has been computed in \cite{Araujo} in general,
without assuming that $(M,\partial M,\bar{g})$ is Ricci-flat with totally geodesic boundary.

As we have seen in section \ref{section_example}, if
 $T^n$  is the  $n$-dimensional torus equipped with
the flat metric $g_0$, then $[0,1]\times T^n$ equipped with the product metric
$dt^2+g_0$ is flat and has geodesic boundary.
Combining this with Theorem \ref{rigidity}, we have the following
rigidity result:

\begin{theorem}
Consider $\tilde{M}=[0,1]\times T^n$ equipped with the product metric
$\tilde{g}=dt^2+g_0$, where $T^n$  is the  $n$-dimensional torus equipped with
the flat metric $g_0$. If $g$ is sufficiently closed to $\tilde{g}$ such that\\
(i) $R_g\geq 0$ in $\tilde{M}$ and $H_\gamma\geq 0$ on $\partial \tilde{M}$,\\
(ii) $g$ and $\tilde{g}$ induce the same metric on $\partial \tilde{M}$,\\
then $g$ is also flat and has totally geodesic boundary.
\end{theorem}

\section*{Acknowledgement}

The authors would like to thank Prof. Yueh-Ju Lin for answering questions on her paper. Part of the work was done when the first author was visiting National Center for Theoretical Sciences in Taiwan, and he is grateful for the kind hospitality. The first author is supported by Basic Science Research Program through the National Research Foundation of Korea (NRF) funded by the Ministry of Education, Science and Technology (2019041021), and by Korea Institute for Advanced Study (KIAS) grant funded by the Korea government (MSIP). The second author is supported by Ministry of Science and Technology, Taiwan, with the grant number: 108-2115-M-024-007-MY2.

\bibliographystyle{amsplain}

\end{document}